\documentclass[a4paper,10pt,reqno, english]{amsart}

\usepackage{amsmath,amssymb,amscd,amsthm,amsfonts}
\usepackage{graphicx,subcaption}
\usepackage{hyperref}
\usepackage{dsfont}
\usepackage{xcolor}
\usepackage{color,soul}
\usepackage[utf8]{inputenc}
\usepackage[nobysame, alphabetic]{amsrefs}

\newtheorem{theorem}{Theorem}[section]
\newtheorem{corollary}[theorem]{Corollary}
\newtheorem{lemma}[theorem]{Lemma}

\newtheorem{question}[theorem]{Question}
\newtheorem{problem}[theorem]{Problem}

\theoremstyle{definition}

\newtheorem{definition}[theorem]{Definition}

\newcommand{\R}{\mathds{R}}
\newcommand{\rr}{\mathds{R}}
\newcommand{\zz}{\mathds{Z}}

\title[General Yao--Yao partition and central transversal theorem]{Generalizations of the Yao--Yao partition theorem and the central transversal theorem}

\author[Manta]{Michael N. Manta}\address{California Institute of Technology, Pasadena, CA 91125} 
\email{mmanta@caltech.edu}

\author[Sober\'on]{Pablo Sober\'on}\address{Baruch College, City University of New York, New York, NY 10010} 
\email{pablo.soberon-bravo@baruch.cuny.edu}

\thanks{Sober\'on's research is supported by NSF grant DMS 2054419 and a PSC-CUNY TRADB52 award.}

\begin{document}

\begin{abstract}
We generalize the Yao--Yao partition theorem by showing that for any smooth measure in $R^d$ there exist equipartitions using $(t+1)2^{d-1}$ convex regions such that every hyperplane misses the interior of at least $t$ regions. In addition, we present tight bounds on the smallest number of hyperplanes whose union contains the boundary of an equipartition of a measure into $n$ regions. We also present a simple proof of a Borsuk--Ulam type theorem for Stiefel manifolds that allows us to generalize the central transversal theorem and prove results bridging the Yao--Yao partition theorem and the central transversal theorem.
\end{abstract}

\maketitle

\section{Introduction}

Mass partition problems study how one can split finite sets of points or measures in Euclidean spaces. They connect topological combinatorics and computational geometry \cites{matousek2003using, Zivaljevic:2017vi, roldan2021survey}. We say that a finite family $P$ of subsets of $\rr^d$ is a \textit{convex partition} of $\rr^d$ if the union of the sets is $\rr^d$, the interiors of the sets are pairwise disjoint, and each set is closed and convex. For a finite measure $\mu$ in $\R^d$, we say that a convex partition of $\rr^d$ is an \textit{equipartition} of $\mu$ if each set in $P$ has the same $\mu$-measure.

 In 1985, Yao and Yao proved the following theorem, motivated by applications in geometric range queries \cite{yao1985general}.

\begin{theorem}[Yao and Yao 1985]\label{thm:yaoyao}
For any finite measure $\mu$ in $\R^d$ that is absolutely continuous with respect to the Lebesgue measure, there exists a convex equipartition of $\R^d$ into $2^d$ regions such that any hyperplane avoids the interior of at least one region.
\end{theorem}

For $d=2$ these partitions are made by two lines, but they become much more involved in higher dimensions. In the original proof, the measure has to be the integral of a twice differentiable positive density function. Lehec \cite{lehec2009yao} found an alternate proof that weakened the condition to every hyperplane having measure 0. In addition to the original applications by Yao and Yao, the discrete version of Theorem \ref{thm:yaoyao} has been applied to geometric Ramsey questions \cite{alon2005crossing}.

The proof technique for the Yao--Yao partition theorem is quite different from standard results in mass partition problems. The question of how far the technique Yao and Yao used can be pushed has been relatively unexplored. In 2014, Rold\'an-Pensado and Sober\'on extended the theorem so that any hyperplane avoids two convex regions and proved a general upper bound in the case of avoiding $t$ convex regions in $\R^2$ \cite{roldan2014extension}. They also provide improved asymptotic bounds when $d$ is fixed and $t$ tends to infinity. The following problem remains open.

\begin{problem}\label{problem:generalizedyao}
Let $t, d$ be positive integers. Find the smallest value $n$ such that for any finite measure $\mu$ in $\R^d$ that is absolutely continuous with respect to the Lebesgue measure there exists a convex equipartition of $\mu$ into $n$ parts such that every hyperplane avoids the interior of at least $t$ regions.
\end{problem}

Currently, the best bound for fixed $t$ is the naive bound that comes from Theorem \ref{thm:yaoyao}. Take any Yao--Yao partition in $\R^d$ and partition each cell using $t-1$ parallel hyperplanes. This gives us a convex equipartition such that any hyperplane avoids the interior of at least $t$ regions, so $n \leq t\cdot2^{d}$. In this paper, we improve the naive bound with the following theorem that generalizes Theorem \ref{thm:yaoyao} and one of the main results from \cite{roldan2014extension}.

\begin{theorem}\label{thm:generalyao}
Let $t,d$ be positive integers. For any finite measure $\mu$ in $\R^d$, absolutely continuous with respect to the Lebesgue measure, there exists a convex partition of $\rr^d$ into $(t+1)2^{d-1}$ regions of equal $\mu$-measure such that every hyperplane avoids the interior of at least $t$ regions.
\end{theorem}

This bound is exact in the case of $d=1$ and matches the known bounds for $d=2$ and any $t$ and for $t=1,2$ and any $d$ \cite{roldan2014extension}. 
Problem \ref{problem:generalizedyao} focuses on hyperplane transversals, so a natural question to ask is for the smallest number of regions needed to avoid all affine subspace transversals of other dimensions. In this paper, we study the case of lines. 

Because a line can pass through a hyperplane at most once, it is sufficient to obtain partitions whose boundaries are contained in the union of few hyperplanes. Some existing mass partitions problems exhibit convex equipartitions with few hyperplanes containing all boundaries.  One classic example is the Gr\"unbaum--Hadwiger--Ramos problem \cites{grunbaum1960partitions, Blagojevic:2018jc}, where the aim is to split simultaneously as many measures as possible into $2^k$ equal parts using $k$ hyperplanes. Another is a recent conjecture by Langerman \cites{barba2019sharing, Hubard:2019we}, which claims that \textit{any $dk$ measures can be simultaneously split into two equal parts by a chessboard coloring induced by $k$ hyperplanes}. The key difference with the problem we discuss here is that we don't require the hyperplanes to extend indefinitely.

\begin{problem}\label{problem:smallestkforn}
Given positive integers $n, d$, find the smallest integer $k$ such that any finite measure in $\R^d$ absolutely continuous with respect to the Lebesgue measure can be partitioned into $n$ convex regions of equal measure whose boundaries are contained in the union of at most $k$ hyperplanes.
\end{problem}

Using the classic ham sandwich theorem in a recursive argument we obtain the following bounds.

\begin{theorem}\label{thm:smallestkforn}
Let $n,d$ be positive integers. The following bounds hold for Problem \ref{problem:smallestkforn}
\[
\left \lceil \frac{n-1}{d} \right \rceil \leq k \leq \frac{n}{d} + (d-1)\log_2\left( \frac{n}{d}\right) + d-2.
\]
\end{theorem}

For a fixed dimension, these bounds imply that $k = n/d +O(\log (n))$.  One consequence is the following Yao--Yao type corollary.

\begin{corollary}\label{coro-line-miss-many}
    Let $d$ be a fixed positive integer. For any positive integer $n$ and any measure $\mu$ in $\rr^d$ absolutely continuous with respect to the Lebesgue measure, there exists a convex partition of $\rr^d$ into $n$ regions of equal $\mu$-measure such that every line misses the interior of at least $(d-1)n/d - O(\log n)$ regions.
\end{corollary}

Using a generalization of the ham sandwich theorem for ``well-separated'' families proved by B\'ar\'any, Hubard, and Jer\'onimo \cite{barany2008slicing}, we can improve this bound even further in the cases of $d=2$ and $d=3$ 

\[
    \left \lceil \frac{n-1}{d} \right \rceil \leq k \leq \frac{n}{d} + O(1)  .
\]

This sharper bound mainly solves these two cases as the $O(1)$ term is bounded by $1/2$ and $13/3$ for $d=2$ and $d=3$, respectively. We present multiple constructions that exhibit each of these bounds, each with different geometric properties.

Finally, we study the connection of Yao--Yao partitions with yet another generalization of the ham sandwich theorem. The following result, proved independently by Dolnikov \cite{dol1992generalization} and \v{Z}ivaljevi\'c and Vre\'cica \cite{zivaljevic1990extension}, is known as the central transversal theorem.

\begin{theorem}[Central transversal theorem]\label{thm:ctt}
Let $k, d$ be non-negative integers such that $k \leq d - 1$. For any set of k+1 absolutely continuous probability measures $\mu_{1}, \ldots, \mu_{k+1}$ there exists a $k$-dimensional affine space $V$ such that for any closed half-space $H$ with $V \subset H$ we have $\mu_{i}(H) \geq \frac{1}{d-k+1}$ for $i=1,\ldots, k+1$.
\end{theorem}

The known proofs of this result involve computing non-trivial topological invariants of some associated spaces, such as Stiefel--Whitney characteristic classes or the Fadell--Husseini index. We first show a new proof that uses a simple homotopy argument. This proof method also allows us to prove some new results, such as the following theorem.

\begin{theorem}\label{thm:two-hyperplanes}
	Let $\mu_1, \ldots, \mu_{d}$ be $d$ probability measures in $\rr^d$.  Then, we can find two hyperplanes $H_1, H_2$ such that
	\begin{itemize}
		\item $H_1 \cup H_2$ splits $\mu_1$ into four equal parts,
		\item each half-space $H$ that contains $H_1 \cap H_2$ satisfies $\mu_i(H) \ge \frac{1}{3}$ for $i=2,3,\ldots, d-1$, and
		\item $H_1$ splits $\mu_d$ into two equal parts.
	\end{itemize}
\end{theorem}

If we project onto the orthogonal complement of $H_1 \cap H_2$, the first equipartition is a Yao--Yao partition. The subspace $H_1 \cap H_2$ is also a central transversal for the next $d-1$ measures as in Theorem \ref{thm:ctt}.

We describe similar results, such as Theorem \ref{thm:yao-transversal}, that interpolate between higher-dimensional Yao--Yao partitions and the central transversal theorem.

\subsection{Structure of the paper}

In Section \ref{sect:toolsnotation}, we begin by introducing notation that we use throughout the paper regarding Yao--Yao partitions and general properties of such partitions. Within this section, we also provide an overview of the proof technique used for the Yao--Yao partition theorem. We use these tools in Section \ref{sect:proofgeneralyao} to prove Theorem \ref{thm:generalyao}. In Section \ref{sect:linetransversal}, we prove our bounds for Problem \ref{problem:smallestkforn}. Section \ref{sect:stiefelyao} introduces results that connect Stiefel manifolds and Yao--Yao partitions. Lastly, we present remarks and open problems in Section \ref{sect:remarks}.

\section{Tools and notation}\label{sect:toolsnotation}

Let $\mu$ be a measure in $\rr^d$. We say that $\mu$ is \textit{absolutely continuous} if it is absolutely continuous with respect to the Lebesgue measure in $\rr^d$ and $\mu(A) \neq 0$ for every open $A$ subset of $\rr^d$. We say that $\mu$ is \textit{finite} if $\mu(\rr^d) < \infty$.

Before we construct the convex partitions of $\rr^d$ for our main results, we define Yao--Yao partitions and provide an overview of the original proof, which contains tools we use to prove Theorem \ref{thm:generalyao}.

A \textit{Yao--Yao partition} is a convex partition of a $\R^d$ into $2^{d}$ regions, defined by a recursive process. A key property of these partitions is that every hyperplane misses the interior of at least one region. Yao--Yao partitions are determined by the choice of an ordered orthonormal basis $(u_1, u_2, \ldots, u_d)$. Each Yao--Yao partition $P$ has a point, denoted $C(P)$, which we call the \textit{center of the partition}.  

\begin{definition}
Let $(u_1, u_2, \ldots, u_d)$ be an orthonormal basis of $\rr^d$. We say a hyperplane is \textit{horizontal} if it is orthogonal to $u_d$. For a horizontal hyperplane $H$, we define the open half-spaces induced by $H$ as $H_{+} = \{x + tu_{d} \mid x \in H, t > 0\}$ and $H_{-} = \{x + tu_{d} \mid x \in H, t < 0\}$. Similarly, for a measure $\mu$ we define $\mu_{+}$ as $\mu$ restricted to $H_{+}$ and $\mu_{-}$ as $\mu$ restricted to $H_{-}$. For a vector $v = (v_1, v_2, \ldots, v_d) \in S^{d-1}$ such that $v_d \neq 0$, the \textit{projection onto $H$ in the direction of $v$} is a mapping from $\rr^d$ to $H$ denoted by $p_{v}(\cdot)$. Denote the projection of $\mu_+$ and $\mu_-$ onto $H$ by $p_v(\mu_+)$ and $p_v(\mu_-)$, respectively. 
\end{definition}

Let $p$ be a point and $(u_1, u_2, \ldots, u_d)$ be an orthonormal basis of $\rr^d$. A Yao--Yao partition $P$ of $\rr^d$ induced by the orthonormal basis $(u_1, u_2, \ldots, u_d)$ with center $p$ is a partition obtained in the following way. For $d=1$, it is the partition of $\rr^1$ into two infinite rays in opposite directions starting from $p$. For $d>1$, we take $H$ to be the horizontal hyperplane through $p$.  Now, we construct two Yao--Yao partitions $P_+$ and $P_-$ of $H$ induced by the orthonormal basis $(u_1, \ldots, u_{d-1})$, each with center $p$ and a vector $v=(v_1, \ldots, v_d)$ such that $v_d \neq 0$. We take $P=\{p_v^{-1}({C})\cap H_+ : C \in P_+\}\cup \{p_v^{-1}({C})\cap H_- : C \in P_-\}$ as our final partition.

Given the center and the orthonormal basis, a Yao--Yao partition is determined by a binary tree of height $d-1$ of projection vectors. Yao and Yao proved the following theorem.

\begin{theorem}[Yao, Yao 1985]
    Let $d$ be a positive integer, $\mu$ be a finite absolutely continuous measure in $\rr^d$, and $(u_1, \ldots, u_d)$ be an orthonormal basis of $\rr^d$. There exists a unique Yao--Yao partition $P$ of $\rr^d$ induced by $(u_1, \ldots, u_d)$ such that
    \[
    \mu(C) = \frac{\mu(\rr^d)}{2^d} \qquad \mbox{ for every }C \in P.
    \]
\end{theorem}

If the orthonormal basis is fixed, then the center depends only on the measure, so we denote it as $C(\mu)$. We call this partition the Yao--Yao equipartition of $\mu$ or the Yao--Yao partition of $\mu$ if there is no risk of confusion. Each cell of the partition is a cone with apex $C(\mu)$. In dimension $d=1$, the center is the point that bisects the measure. For dimensions $d \geq 2$, we bisect the measure with a horizontal hyperplane $H$ and find a projection vector $v$ such that $C(p_{v}(\mu_{+})) = C(p_{v}(\mu_{-}))$. Yao and Yao proved that there is a unique projection vector $v$ for which this happens. Moreover, as $\mu$ varies continuously, so does $C(\mu)$.

The reason why $C(p_{v}(\mu_{+}))$ and $C(p_{v}(\mu_{-}))$ coincide for some projection vector is that otherwise we can construct a map from $B^{d-1}$, the $(d-1)$-dimensional unit ball, thought of as the set of points in $S^{d-1}$ with non-negative last coordinate to $S^{d-2}$ as
\begin{align*}
    f:B^{d-1} & \to S^{d-2} \\
    v &\mapsto \begin{cases}
    \frac{C(p_{v}(\mu_{+})) - C(p_{v}(\mu_{-}))}{||C(p_{v}(\mu_{+})) - C(p_{v}(\mu_{-}))||} & \mbox{if } v_d >0, \\
    -v & \mbox{if } v_d = 0.
    \end{cases}
\end{align*}

Yao and Yao showed that this map is well defined and continuous if the top and bottom centers never coincide. This is a contradiction since it would be imply that $f|_{S^{d-2}}$ has degree zero instead of $(-1)^{d-1}$ (Yao and Yao instead finish the proof using the Borsuk--Ulam theorem).

Additional geometric arguments show that there is only one projection vector up to scalar multiples that makes the top and bottom center coincide. To be precise, if the last coordinate of $v$ is positive and we choose a new projection vector $v' = v + h$ where $h$ is parallel to $H$ and not zero, then
\begin{align*}
    \langle C(p_{v'}(\mu_{+})), h \rangle < \langle C(p_{v}(\mu_{+})), h \rangle  \\
     \langle C(p_{v'}(\mu_{-})), h \rangle > \langle C(p_{v}(\mu_{-})), h \rangle
\end{align*}
where $\langle \cdot, \cdot \rangle$ denotes the standard dot product.  In other words, as the projection direction changes, the centers move in opposite directions.

The first family of partitions we consider are based on Yao--Yao partitions.

\begin{definition}
Let $\mu$ be an absolutely continuous finite measure in $\rr^d$ and $(u_1, \ldots, u_d)$ be an orthonormal basis of $\rr^d$. For $\alpha, \beta > 0$ such that $\alpha + \beta = 1$ we define recursively an $(\alpha,\beta)$-partition of $\mu$ with $(\alpha, \beta)$-center in the following way. In dimension $d=1$, the $(\alpha, \beta)$-center $C(\mu)$ is the point such that the regions in the directions of $-u_1$ and $u_1$ have measure $\alpha(\mu(\R))$ and $\beta(\mu(\R))$, respectively. For dimensions $d \geq 2$, we first halve the measure with a horizontal hyperplane $H$ and find a projection vector $v$ such that $C(p_{v}(\mu_{+})) = C(p_{v}(\mu_{-}))$. In the previous equality, we are considering $C(p_{v}(\mu_{+}))$ the center of the $(\alpha, \beta)$-partition $P_+$ of $p_v(\mu_{+})$ with respect to the basis $(u_1, \ldots, u_{d-1})$ and $C(p_{v}(\mu_{-}))$ the center of the $(\alpha, \beta)$-partition $P_-$ of $p_v(\mu_{-})$ with respect to the basis $(u_1, \ldots, u_{d-1})$. Finally, we define the $(\alpha, \beta)$-partition as $P=\{p_v^{-1}({C})\cap H_+ : C \in P_+\}\cup \{p_v^{-1}({C})\cap H_- : C \in P_-\}$.
\end{definition}

Note that in the case $\alpha = \beta = \frac{1}{2}$, we have a Yao--Yao equipartition. The proof of existence and uniqueness of Yao--Yao partitions can be used verbatim to prove the existence and uniqueness of $(\alpha,\beta)$-partitions. Each $(\alpha, \beta)$-partition of a measure splits $\rr^d$ into $2^d$ convex parts. Given an $(\alpha, \beta)$-partition, every hyperplane misses the interior of at least one of the $2^d$ sections.

The only change we make, relative to Yao and Yao's original construction, is in the first step of the construction. Each subsequent step uses a halving horizontal hyperplane. More general versions of these partitions, where the horizontal hyperplanes are not necessarily halving hyperplanes, were studied by Lehec \cite{lehec2009yao}. As we will show, $(\alpha, \beta)$-partitions have additional useful structural properties.

For a finite absolutely continuous measure $\mu$ in $\rr^d$ and fixed $(\alpha, \beta)$, we introduce a lemma that describes the relationship between (classic) Yao--Yao equipartitions and $(\alpha, \beta)$-partitions of $\mu$.

\begin{lemma}\label{lemma:structure1}
	Let $\alpha, \beta$ be positive real numbers whose sum is $1$ and let $(u_1, \ldots, u_d)$ be an orthonormal basis of $\rr^d$. Given a finite absolutely continuous measure $\mu$ in $\rr^d$, consider the $(\alpha, \beta)$-partition of $\mu$ induced by the basis $u_1, \ldots, u_d$. Then, we can split the regions of the partition into $2^{d-1}$ pairs $(A_i, B_i)$ so that
	\begin{itemize}
		\item $A_i \cup B_i$ is convex for each $i$,
		\item each region $A_i$ contains an infinite ray in the direction of $-u_1$, and
		\item each region $B_i$ contains an infinite ray in the direction of $u_1$.
	\end{itemize}
Additionally, for each $i=1,\ldots,2^{d-1}$ the regions satisfy $\displaystyle \mu(A_i) = \frac{\alpha}{2^{d-1}} \mu(\rr^d)$ and $\displaystyle \mu(B_i) = \frac{\beta}{2^{d-1}} \mu(\rr^d)$.
\end{lemma}

\begin{proof}
We proceed by induction on the dimension $d$. Clearly, the claim is true for $d=1$. Now we assume that the statement is true for $d-1$ and want to show that it is true for $d$. Project $\mu$ onto the horizontal halving hyperplane $H$ in the direction of the associated projection vector $v$ of the $(\alpha, \beta)$-partition of $\mu$. The projected partitions of $p_{v}(\mu_{+})$ and $p_{v}(\mu_{-})$ are both $(\alpha, \beta)$-partitions. Therefore, we can find pairings of regions in $H_{+}$ and $H_{-}$ that satisfy our requirements. Taking the pre-image of these partitions does not change convexity and the same rays in the directions $\pm u_1$ are contained. Therefore, these pairings work for the original $(\alpha, \beta)$-partition. Since $\mu_{+}$ and $\mu_{-}$ each have measure $\displaystyle \frac{1}{2} \mu(\R^{d})$, we know $\displaystyle \mu(A_i) = \frac{\alpha}{2^{d-1}} \cdot \frac{1}{2} \mu(\R^d)$ and $\displaystyle \mu(B_i) = \frac{\beta}{2^{d-1}} \cdot \frac{1}{2} \mu(\R^d)$ for each $i$. Thus, the claim is true for all $d$.
\end{proof}

Throughout the rest of the paper we will use $A_i$, $B_i$ to denote the regions as described above. For each $i$, let $C_i = A_i \cup B_i$. We call $C_1, \ldots, C_{2^{d-1}}$ the \textit{frame} of the partition.  Notice that the frame of the partition is a convex equipartition of $\mu$ into $2^{d-1}$ parts. We now show that the frame of the partition does not depend on the values of $\alpha, \beta$, as it is defined by the projection of $\mu$ onto the orthogonal complement of $u_1$.

\begin{lemma}\label{lemma:structure2}
	Let $\alpha, \beta$ be positive real numbers whose sum is $1$ and let $u_1, \ldots, u_d$ be an orthonormal basis of $\rr^d$ for $d \geq 2$.  Given a measure $\mu$ in $\rr^d$, let $A_1, \ldots, A_{2^{d-1}}, \\ B_1, \ldots, B_{2^{d-1}}$ be the $(\alpha, \beta)$-partition of $\mu$ induced by the basis $u_1, \ldots, u_d$. Let $V$ be the hyperplane orthogonal to $u_1$. Let $\pi: \rr^d \to V$ be the orthogonal projection onto $V$. We denote for each $1 \le i \le 2^{d-1}$ the sets $C_i = A_i \cup B_i$ and $D_i = \pi (C_i)$. Then, $D_1, \ldots, D_{2^{d-1}}$ is the Yao--Yao equipartition of $\pi (\mu)$ on $V$ induced by the basis $u_2, \ldots, u_{d}$.
\end{lemma}

\begin{proof}
We proceed by induction on $d$. Clearly, the claim is true for $d=2$. Now we assume that the statement is true for $d-1$ and want to show that it is true for $d$. Let $R \subset V$ be the affine subspace of dimension $d-2$ that halves $\pi (\mu)$, notice that $R = \pi(H)$. We prove this lemma by using the equivalence of two projections onto $R$.

Let $v$ be the associated projection vector of the $(\alpha, \beta)$-partition. By the inductive hypothesis, we know that the frame of the $(\alpha,\beta)$-partition of $p_v(\mu_+)$ projects onto the Yao--Yao partition of $\pi(p_v(\mu_+))$.  Similarly, the frame of the $(\alpha,\beta)$-partition of $p_v(\mu_-)$ projects onto the Yao--Yao partition of $\pi(p_v(\mu_-))$.  Therefore, the projection of the frame of the $(\alpha, \beta)$-partition of $\mu$ projects onto a Yao--Yao partition of $\pi(\mu)$ whose projection vector is $\pi(v)$.  This is simply because $\pi \circ p_v = p_{\pi(v)} \circ \pi $.  The resulting partition of $\pi(\mu)$ is clearly an equipartition, so we obtain the desired result.
\end{proof}

Now we show that the frame of a partition also remains constant if we restrict the measure to the $A_i$'s or to the $B_i$'s, which will be useful for our definition of multicenter partitions.

\begin{corollary}\label{cor:restriction}
	With the notation of Lemma \ref{lemma:structure2} define $\displaystyle A = \bigcup_{i=1}^{2^{d-1}} A_i$ and $ \mu|_{A}$ the restriction of $\mu$ to $A$.  Then, $D_1, \ldots, D_{2^{d-1}}$ is the Yao--Yao equipartition of $\pi(\mu|_{A})$ on $V$ induced by the basis $u_2, \ldots, u_{d}$.  The lemma also holds if we replace $A$ by $\displaystyle B = \bigcup_{i=1}^{2^{d-1}} B_i$.
\end{corollary}

\begin{proof}
By Lemma \ref{lemma:structure2}, we know that $D_1, \ldots, D_{2^{d-1}}$ is a Yao--Yao partition, and Lemma \ref{lemma:structure1} shows that $\mu(A_i) = \mu|_{A}(C_i)=\pi(\mu|_A)(D_i)$ has the same value for each $i$.  As $\pi(\mu|_A)$ is a finite absolutely continuous measure in $V$, its Yao--Yao equipartition induced by $(u_2, \ldots, u_d)$ is unique.  The proof is analogous for $B$ and $\mu|_B$.
\end{proof}

Given an $(\alpha, \beta)$-partition, each $C_i$ is split into two convex regions $A_i$ and $B_i$.  The boundary between $A_i$ and $B_i$ must therefore be a hyperplane section.  We denote the union of all these boundary pieces the \textit{wings} of the partition.  As we vary $\alpha, \beta$, the frame remains constant while the center and the wings of the partition can change.

Another consequence of the lemmas above is that, for a fixed absolutely continuous finite measure $\mu$ in $\rr^d$, there exists a line $\ell$ parallel to $u_1$ such that
\begin{itemize}
	\item every $(\alpha, \beta)$-partition of $\mu$ has its center on $\ell$,
	\item the boundary of every part $A_i$ of an $(\alpha, \beta)$-partition of $\mu$ contains the infinite ray $\{c - \lambda u_1 : \lambda \ge 0\}$ where $c$ is the center of the $(\alpha, \beta)$-partition, and
	 \item the boundary of every part $B_i$ of an $(\alpha, \beta)$-partition of $\mu$ contains the infinite ray $\{c + \lambda u_1 : \lambda \ge 0\}$ where $c$ is the center of the $(\alpha, \beta)$-partition.
\end{itemize}

We know by the arguments of Yao and Yao that every hyperplane $H$ must avoid the interior of at least one of the regions. Let left and right be the directions defined by $-u_1$ and $u_1$, respectively. If $H\cap \ell$ is a single point $p$, then it must avoid one of the regions $B_i$ if $p$ is at or left of the center $c$ and one of the regions $A_i$ if $p$ is at or right of the center $c$.

The previous results can be improved to give us significant information about the $k$-skeletons of the cones forming and $(\alpha, \beta)$-partition.

\begin{lemma}\label{lem:skeleton}
Let  $k < d$ be positive integers, $(u_1, \ldots, u_d)$ be an orthonormal basis of $\rr^d$ and $\mu$ be a finite absolutely continuous measure.  For any $(\alpha, \beta)$-partition of $\mu$ induced by the basis $(u_1, \ldots, u_d)$, the union of the $k$-skeletons of the parts of the partition contains the translate of $\operatorname{span}\{u_1,\ldots, u_k\}$ that goes through the center of the partition.
\end{lemma}

\begin{proof}
We proceed by induction on $k$.  For $k=1$, the previous results imply the result.  Now assume that $k>1$ and we know the result holds for $k-1$ and any $d \ge k$.  With the notation of the previous lemmas, the frame of the $(\alpha, \beta)$-partition $P$ projects onto $V$ to the Yao--Yao partition $P'$ of $\pi (\mu)$ with basis $(u_2,\ldots, u_d)$.  We know by induction that the union of the $(k-1)$-skeletons of the parts of $P'$ contain the translate of $\operatorname{span}\{u_2,\ldots, u_k\}$.  When we take the inverse of $\pi(\cdot)$ we are extending these boundaries by $\pm u_1$ (one direction with the parts $A_i$ and the other with the parts $B_i$), which finishes the proof.
\end{proof}

If we consider the lemma above with $k=d-1$, we obtain the halving hyperplane at the core of the Yao--Yao construction.  This way one can also argue this lemma by reducing the value of $k$ rather than increasing it.

Since the location of the center $p$ on the line $\ell$ is important, we prove an additional technical lemma to improve Corollary \ref{cor:restriction}.  We want to prove that a center for an $(\alpha',\beta')$-partition of $\mu|_{B}$ is further right than a center for $\mu$.

\begin{lemma}\label{lem:technical-position}
Let $d$ be a positive integer, $(u_1, \ldots, u_d)$ be an orthonormal basis of $\rr^d$, and $\alpha, \alpha', \beta, \beta'$ be positive real numbers such that $\alpha + \beta = \alpha' + \beta' = 1$.  For an absolutely continuous finite measure $\mu$ in $\rr^d$, let $C(\mu)$ be the center of its $(\alpha, \beta)$-partition induced by the basis $(u_1, \ldots, u_d)$, and let $C(\mu|_{B})$ be the center of the $(\alpha',\beta')$-partition of $\mu|_B$.  Then, $C(\mu|_{B})$ is farther right than $C(\mu)$
\end{lemma}

\begin{proof}
    We proceed by induction on $d$.  For $d=1$, the result is clear.  Assume the result holds for $d-1$.  Let $v$ be the projection vector of the $(\alpha, \beta)$-partition of $\mu$, and let $v'$ be the projection vector of the $(\alpha', \beta')$-partition of $\mu|_B$.  We know by the previous lemmas that we can choose scalar multiples of $v$ and $v'$ so that $\pi(v) = \pi(v')$, as the frame of both partitions coincide.  Consider the measures $p_v(\mu_+|_{B})$ and $p_v(\mu_-|_{B})$ on the halving hyperplane $H$.  By the previous lemmas, their $(\alpha',\beta')$-partition shares the frame of $p_v(\mu)$, and their centers $C_1, C_2$ respectively must be to the right of $C(p_v(\mu))=C(\mu)$.  
    
    If $C_1 = C_2$, this means that $v=v'$ and we are done.  If $C_1 \neq C_2$, as we move $v$ to $v'$ the centers $C_1$ and $C_2$ move in opposite directions according to $u_1$ (formally, one of $\langle C_1, u_1\rangle, \langle C_2, u_1\rangle$ increases and the other decreases.  This means that $C(\mu|_{B})$ must be between $C_1$ and $C_2$, and so it is to the right of $C(\mu)$.
\end{proof}

\begin{definition}\label{def:multicenter}
We define a new family of partitions, called \textit{multicenter partitions} with $t$ centers for a measure $\mu$ in $\R^d$ in the following way.
\begin{itemize}
    \item For $t=1$, we take a classic Yao--Yao equipartition of $\mu$.
    \item For $t>1$, let $C_1, \ldots, C_{2^{d-1}}$ be the frame of all possible $(\alpha, \beta)$-partitions of $\mu$.  We first construct a $(1/(t+1), t/(t+1))$-partition of $\mu$ and denote its regions by $A_1, \ldots, A_{2^{d-1}}, B_1, \ldots, B_{2^{d-1}}$.  Let $P$ be a multicenter partition of $\mu|_{B}$ with $t-1$ centers.  Our multicenter partition $Q$ with $t$ centers is defined as
    \[
    Q = \{A_i : i =1,\ldots, 2^{d-1}\}\cup \{K \cap B_i : i = 1,\ldots, 2^{d-1}, K \subset C_i, K \in P\}.
    \]
\end{itemize}

We define the $t$ centers as the union of the center of the $(1/(t+1), t/(t+1))$-partition of $\mu$ and the $t-1$ centers of the multicenter partition $P$ we used in the construction.

\end{definition}

\begin{figure}[ht]
\centering
\includegraphics[scale=0.75]{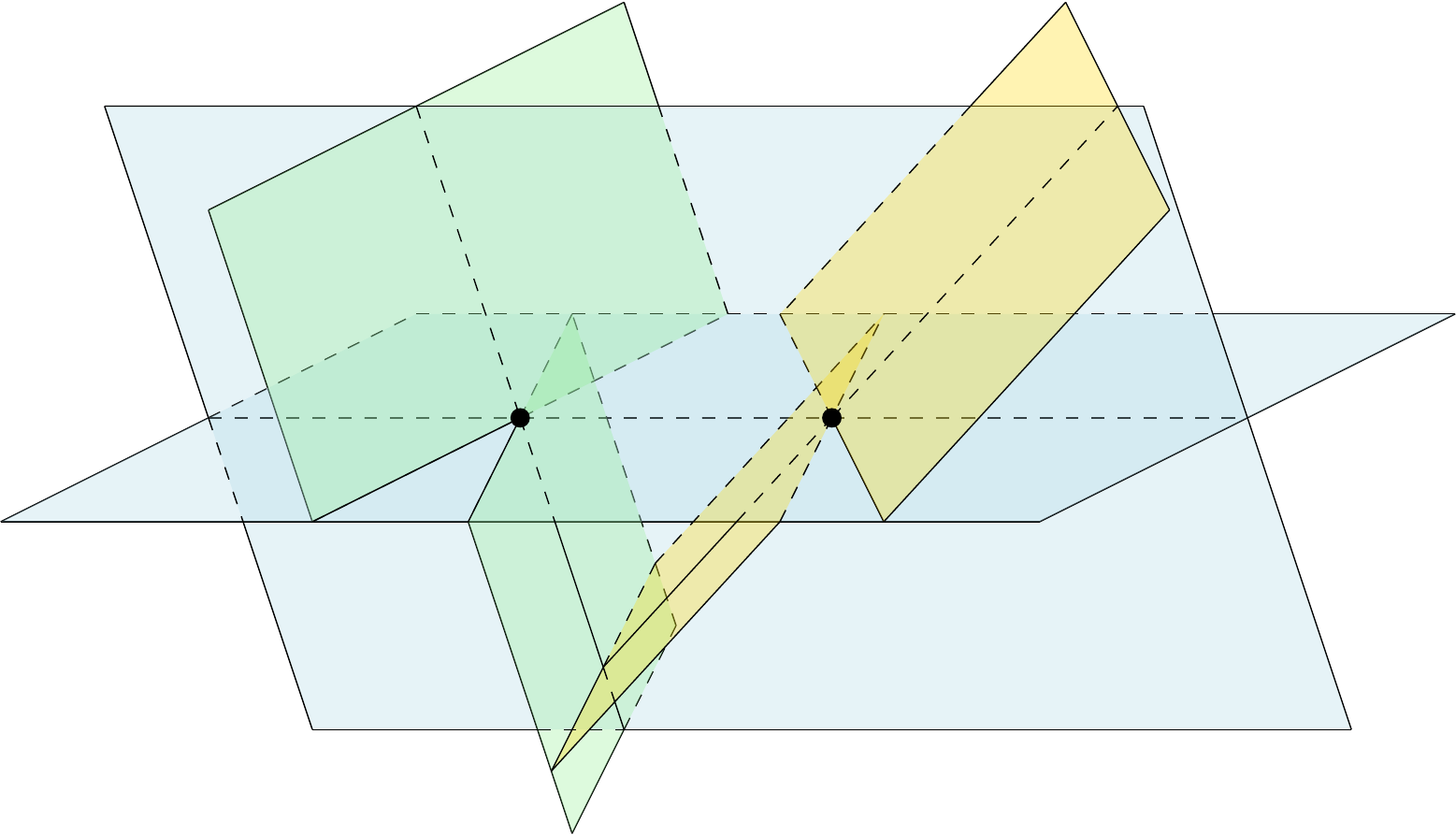}
\caption{Example of a multicenter partition in $\R^3$ with two centers.}
\label{fig:R3convexpartition}
\end{figure}

Each part of a multicenter partition is contained in one of the parts of the frame of $\mu$.  Each region of the frame is iteratively partitioned into $t+1$ convex regions by the wings of each $(\alpha, \beta)$-partition we took.  A multicenter partition with $t$ centers will have exactly $(t+1)2^{d-1}$ convex parts of equal measure.  An example is shown in Figure \ref{fig:R3convexpartition}.  Notice that the subdivision of the right-most center doesn't extend past the wings induced by the first partition (induced by the left-most center).  Now we can prove Theorem \ref{thm:generalyao}.

\section{Proof of Theorem \ref{thm:generalyao}}\label{sect:proofgeneralyao}

 We want to show that for a multicenter partition with $t$ centers every hyperplane misses the interior of at least $t$ regions.  Since a multicenter partition with $t$ centers has $(t+1)2^{d-1}$ regions, all of them convex, we would prove Theorem \ref{thm:generalyao}.  In this section, we use a fixed orthonormal basis $(u_1, u_2, \ldots, u_d)$ and induct on $t$.

\begin{theorem}
Let $\mu$ be a finite absolutely continuous measure in $\R^d$, and let $P$ be a multicenter partition of $\mu$ with $t$ centers.  Every hyperplane misses the interior of at least $t$ regions.
\end{theorem}
\begin{proof}
We proceed inductively.  For $t=1$, this is the result of Yao and Yao.  Assume the result holds for $t-1$.  Let $Q$ be the multicenter partition, with parts labeled as in Definition \ref{def:multicenter}.  Let $C(\mu)$ be the center of the $(1/(t+1),t/(t+1))$-partition of $\mu$, and let $\ell$ be the line with direction $u_1$ through $C(\mu)$.

Let $L$ be a hyperplane.  We first assume that $L \cap \ell$ is a single point.  If $L \cap \ell$ is to the right or coincides with $C(\mu)$, then $L$ misses the interior of one of the regions $A_i$.  By the induction hypothesis, $L$ misses the interior of at least $t-1$ regions of the form $B_{i'} \cap K$ for some $i'$.  If $L \cap  \ell$ is to the left of $C(\mu)$, then $L$ avoids the interior of some set $B_i$.  This means that $L$ avoids all the $t$ parts of the partition of the form $B_i \cap K$.

Any hyperplane $L$ for which $L \cap \ell$ is not a point contains a line parallel to $\ell$.  Therefore, $\pi(L)$, the projection of $L$ onto the subspace orthogonal to $u_1$ is not surjective.  Since the projection of the frame $C_1, \ldots, C_{2^{d-1}}$ is a Yao--Yao partition, this means that $L$ misses the interior of one of the $C_i$.  With this, we have that $L$ misses the interior of the $t+1$ regions in the subdivision of that $C_i$.
\end{proof}

The proof above does not use Lemma \ref{lem:technical-position}.  However, that lemma can help us get intuition regarding which regions of a multicenter partition are avoided.  If $L \cap \ell$ is between the $j$-th and the $(j+1)$-th center of the partition, we can guarantee $L$ avoids $j$ ``left regions $A_i$'' in the recursive definition and $t-j$ right regions which at some point were of the form $B_i \cap K$ in the recursive definition.  The union of the right regions avoided is convex.  An example is shown in Figure \ref{sect:linetransversal}.  In that figure, one of the avoided regions is shaded.  We can see that the section right of the shaded region is also missed by the hyperplane, and their union is convex.

\begin{figure}[ht]
\centering
\includegraphics[scale=0.75]{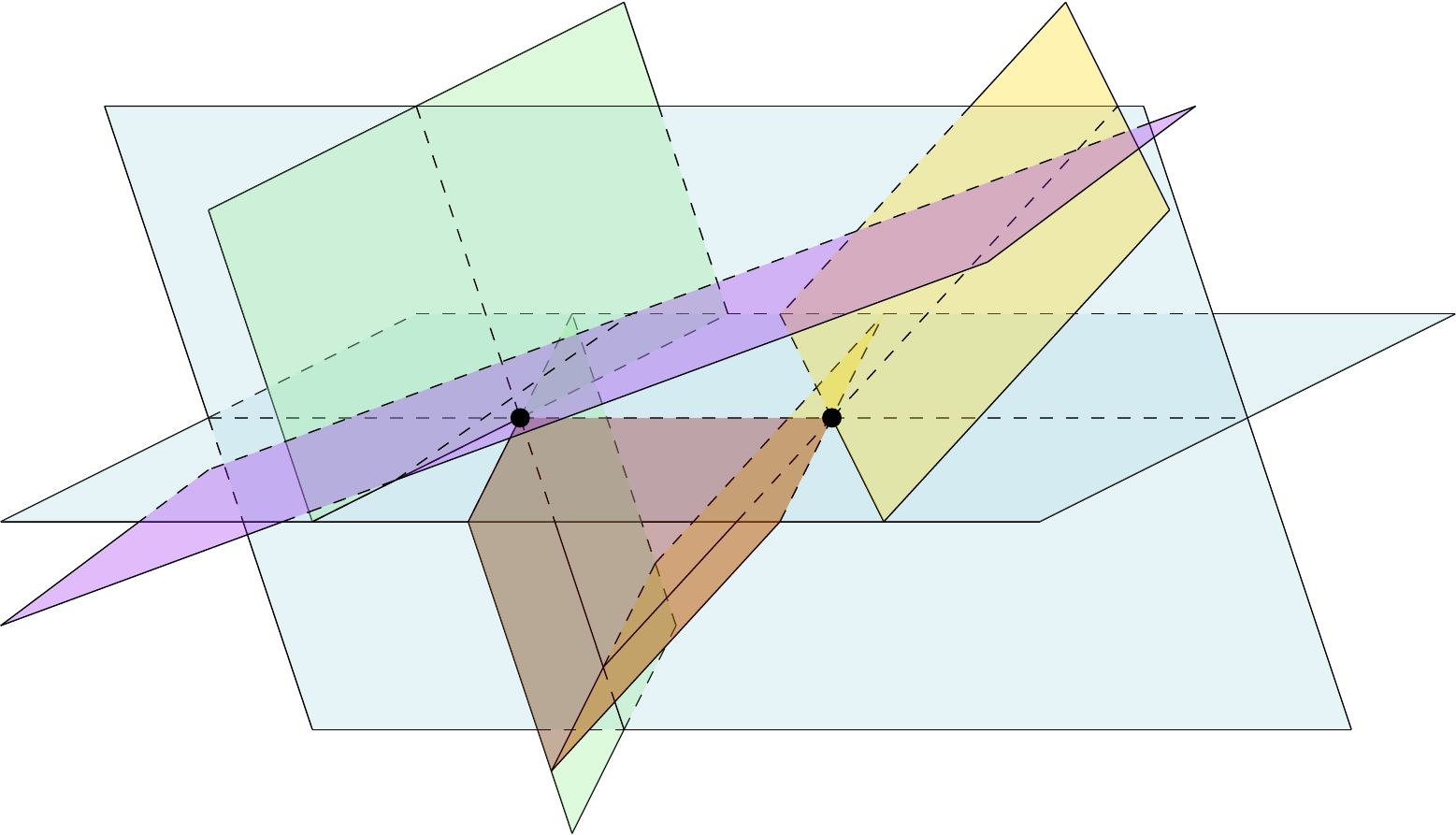}
\caption{Hyperplane transversal for a multicenter partition in $\R^3$}
\label{fig:R3avoidance}
\end{figure}

\section{Line transversal problem}\label{sect:linetransversal}

We first introduce the notion of complexity for a partition. The \textit{complexity} of a convex partition refers to the smallest number of hyperplanes whose union contains all the boundaries between parts of the partition.

If $P$ is a partition into $n$ parts with complexity $k$ then every line misses the interior of at least $n - (k+1)$ regions, so we care about minimizing the complexity to maximize the number of parts missed by any line. The partitions we use involve iterated partitions by successive hyperplanes.  These are similar to those used recently for high-dimensional versions of the necklace splitting problem \cites{deLongueville:2006uo, Karasev:2016cn, Blagojevic:2018gt}.  The key difference is that we wish each part of the partition to have the same size, as opposed to distribute the parts among a fixed number of participants.

To obtain a general bound for the minimum complexity needed to find an equipartition of any measure, we use the ham sandwich theorem.  This was first proved by Steinhaus, who attributed the result to Banach \cite{Steinhaus1938}.

\begin{theorem}[Ham sandwich theorem]\label{thm:hamsandwich}
Given $d$ finite measures in $\R^d$, each absolutely continuous with respect to the Lebesgue measure, there exists a hyperplane that divides $\R^d$ into two half-spaces of the same size with respect to each measure.
\end{theorem}

With this tool, we can introduce a lemma that allows us to save many hyperplanes as the dimension increases. 

\begin{lemma}\label{lem:d2xlemma}
Let $d$ and $x$ be positive integers.  For any absolutely continuous measure $\mu$ in $\R^d$ there exists a convex equipartition into $d2^{x}$ sets of complexity $(d-1) + 2^{x} - 1$.
\end{lemma}
\begin{proof}
First, we partition the measure into convex regions of equal measure using $d-1$ parallel hyperplanes. Now we can treat the $d$ regions as distinct measures. By the ham sandwich theorem, we can bisect all the regions using a hyperplane. We repeat this process recursively for the regions on the positive side of the bisecting hyperplane and for the negative side of the bisecting plane until we have regions of measure $\mu(\R^d)/d2^{x}$. The total number of hyperplanes used is $2^x - 1$ aside from the initial $d-1$ hyperplanes. Hence, we use a total of $(d-1) + 2^{x} - 1$ to partition the measure into $d2^{x}$ regions of equal measure.
\end{proof}

Because we are using a single hyperplane to bisect multiple regions at once, we can save many hyperplanes and obtain the following bound. We assume we have a fixed orthonormal basis $\{u_1, u_2, \ldots, u_d\}$. We call hyperplane vertical if it is orthogonal to $u_1$. We are now ready to prove Theorem \ref{thm:smallestkforn}.

\begin{proof}[Proof of Theorem \ref{thm:smallestkforn}]
We first prove the upper bound.  We can recursively define how we partition $\mu$ into $n$ regions of equal measure in the following way. Any positive integer $n$ can be expressed in the form $d2^{\alpha_1} + d2^{\alpha_2} + d2^{\alpha_3} + \ldots + \epsilon$ such that $\alpha_{j} > \alpha_{k}$ for $j < k$ and $\epsilon < d$. We can use a vertical hyperplane to partition $\mu$ such that the left side of the hyperplane has measure $(\mu(\R^d)d2^{\alpha_1})/n$. Using another vertical hyperplane to partition the right region we repeat this process so that the section between the two hyperplanes has measure $(\mu(\R^d)d2^{\alpha_2})/n$ and so on for each $\alpha_{i}$. 

We repeat this for at most ${\alpha_1} \leq \log_{2}(n/d)$ terms. Once the right-most side of the partition has measure less than $(\mu(\R^d)d)/n$, we use $\epsilon - 1$ parallel hyperplanes to partition the region into convex regions of equal measure. In each region with measure $(\mu(\R^d)d2^{\alpha_i})/n$, we can apply Lemma \ref{lem:d2xlemma} to partition that region into $d2^{\alpha_i}$ regions of equal measure using $(d-1) + 2^{\alpha_i} - 1$ hyperplanes. Therefore, we use at most $k \leq (d-1)\log_{2}(n/d) + (n/d - \epsilon/d) + \epsilon - 1$ hyperplanes and consequently $k = n/d + O(\log(n))$.

We can use a measure concentrated around the moment curve to obtain a lower bound for the complexity of any convex partition of $\rr^d$. A single hyperplane can intersect a moment curve at no more than $d$ points and contribute at most $d+1$ convex regions in which each piece of the measure can lie. Each subsequent hyperplane contributes at most an additional $d$ regions. This implies that, for a convex equipartition of $\mu$ into $n$ regions with complexity $k$, we have $n \leq kd + 1$. Therefore the lower bound for the complexity is $(n-1)/d \leq k$.
\end{proof}

An intuitive way to look at the proof above is that every time we use use Lemma \ref{lem:d2xlemma}, we get $d-1$ hyperplanes above the $n/d$ ideal bound.  This is done once for every $1$ in the binary representation of $\lfloor n/d \rfloor$.
For $d=2$ and $d=3$ we can find much simpler partitions that yield the following sharper bound.

\[
    \left \lceil \frac{n-1}{d} \right \rceil \leq k \leq  \frac{n}{d} + O(1) 
\]

\smallskip

The bounds largely solve the problem in those dimensions, as the term $O(1)$ is bounded by $1/2$ for $d=2$ and by $13/3$ for $d=3$.  We prove this by using a generalization of the ham sandwich theorem for ``well-separated'' families proved by B\'ar\'any, Hubard, and Jer\'onimo \cite{barany2008slicing}.  We say that a family $\mathcal{F}$ of subsets of $\rr^d$ is well-separated if for every $\mathcal{A}\subset\mathcal{F}$ we have that $\bigcup \mathcal{A}$ can be separated from $\bigcup (\mathcal{F}\setminus \mathcal{A})$ by a hyperplane.

\begin{theorem}[Bárány, Hubard, Jerónimo 2008]\label{thm:bhj}
Let $d$ be a positive integer.  For all $i=1,\ldots, d$, let $\mu_{i}$ be a finite measure on $\R^d$, absolutely continuous with respect to the Lebesgue measure, with support $K_{i}$ for all $i \in \{1, 2, \dots, d\}$. Assume the family $\mathcal{F} = \{K_{1},\dots,K_{d}\}$ is well-separated and let $\alpha = (\alpha_{1},\dots,\alpha_{d}) \in (0,1)^{d}$. Then there exists a half-space, $H$, such that $\mu_{i}(K_{i} \cap H) = \alpha_{i} \cdot \mu_{i}(K_{i})$, for every $i \in \{1, 2, \dots, d\}$.
\end{theorem}

B\'ar\'any et al also determined conditions that guarantee the uniqueness of the hyperplane above, but we do not require it for our proof.  Using Theorem \ref{thm:bhj}, we introduce alternative partitions that yield better upper bounds than Theorem \ref{thm:smallestkforn} for $d=2$ and $d=3$.

\subsection{d=2}
Let $n = 2q + r$ for non-negative integers $q$ and $r < 2$. If $n < 2$, then $n=1$, so we use no hyperplanes. For $n \geq 2$, we first use $r$ vertical lines to partition the measure such that each of the $r$ left-most regions has measure $\mu(\R^2)/n$. Then we bisect the rightmost region with a line. The two regions induced are well-separated, so we can apply Theorem \ref{thm:bhj} to the pair of regions. We can use $q-1$ lines to partition the pair such that the positive half-space of each line has measure $\mu(\R^2)/q$ in each region of each pair. Therefore, for $n \geq 2$ the number of lines we use is 
\[
\frac{n}{2} + \frac{r}{2}.
\]

We refer to this partition as a B\'ar\'any, Hubard, and Jer\'onimo partition (see Figure \ref{fig:BHJdiagram}). Note that $r \leq 1$, so the complexity is bounded by $1/2$.

The second partition we introduce does not use Theorem \ref{thm:bhj} but follows immediately from Theorem \ref{thm:generalyao}. If the number of regions $n$ is odd, we use a vertical line to partition the measure so the left side has measure $\mu(\R^2)/n$ and then we take a multicenter equipartition for the right side of the line. In the case that $n$ is even, we can take a multicenter equipartition. The number of lines used in this partition is 
\[
k = \Bigl\lceil{\frac{n}{2}\Bigr\rceil}.
\]

We call these partitions \textit{modified multicenter partitions} (see Figure \ref{fig:multicenterR2diagram}). Note that this is the same bound proven by 
Rold\'an-Pensado and Sober\'on \cite{roldan2014extension}. The construction they provide is different because they use a rotating half-space to partition the measure (see Figure \ref{fig:RPSdiagram}). We refer to this partition as a Rold\'an-Pensado and Sober\'on partition. The differences between the partitions lies in their geometry. The modified multicenter partition iterates on the right regions, the Rold\'an-Pensado and Sober\'on partition uses rotating cuts, and the B\'ar\'any, Hubard, and Jer\'onimo partition preserves one cut in any direction and iteratively partitions the pair of regions.

\begin{figure}[ht]
     \centering
     \begin{subfigure}[b]{0.3\textwidth}
            \centering
            \includegraphics[width=\textwidth, height=2.5cm]{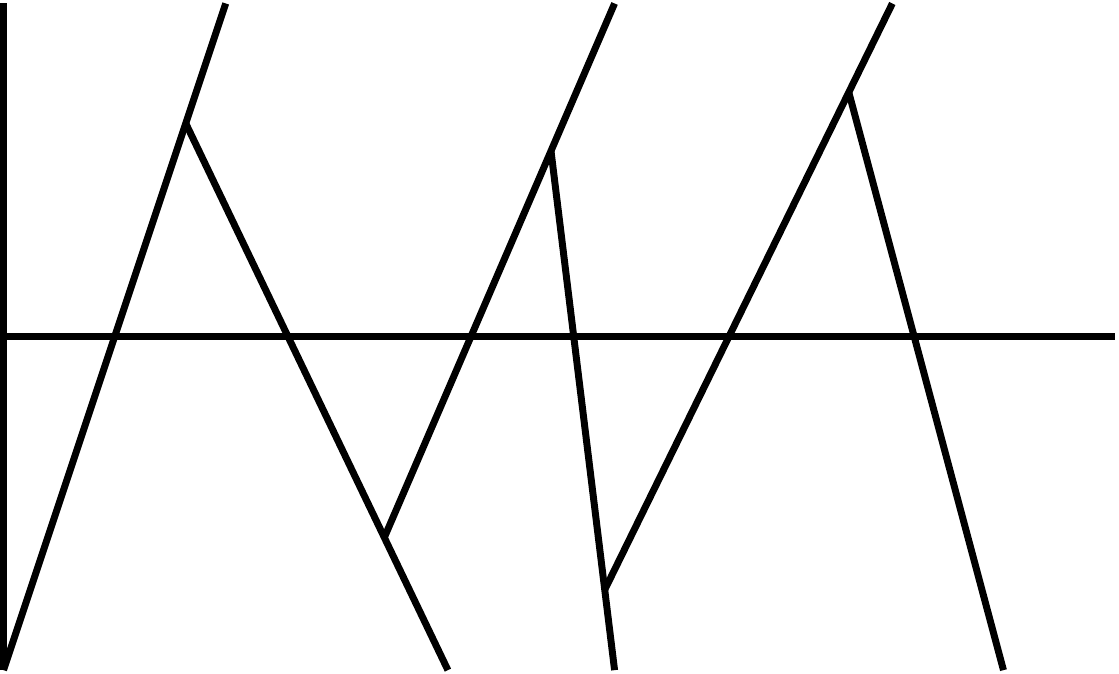}
            \caption{Modified multicenter partition in $\R^2$}
            \label{fig:multicenterR2diagram}
     \end{subfigure}
     \hfill
     \begin{subfigure}[b]{0.3\textwidth}
         \centering
            \includegraphics[width=\textwidth, height=2.5cm]{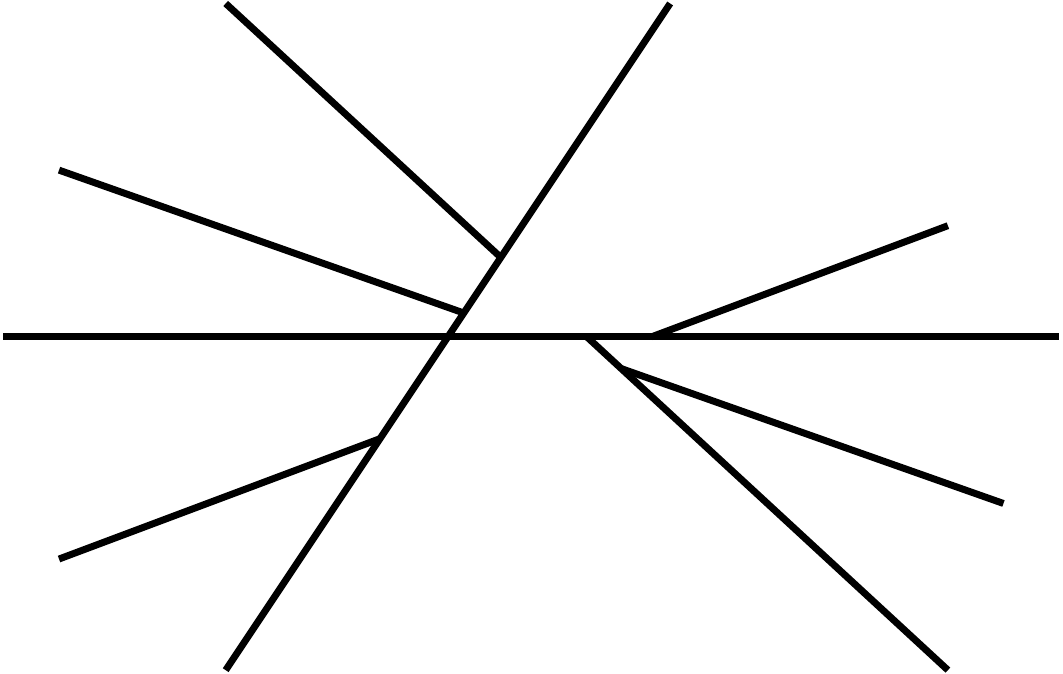}
            \caption{Rold\'an-Pensado and Sober\'on partition}
            \label{fig:RPSdiagram}
     \end{subfigure}
     \hfill
     \begin{subfigure}[b]{0.3\textwidth}
         \centering
         \includegraphics[width=\textwidth, height=2.5cm]{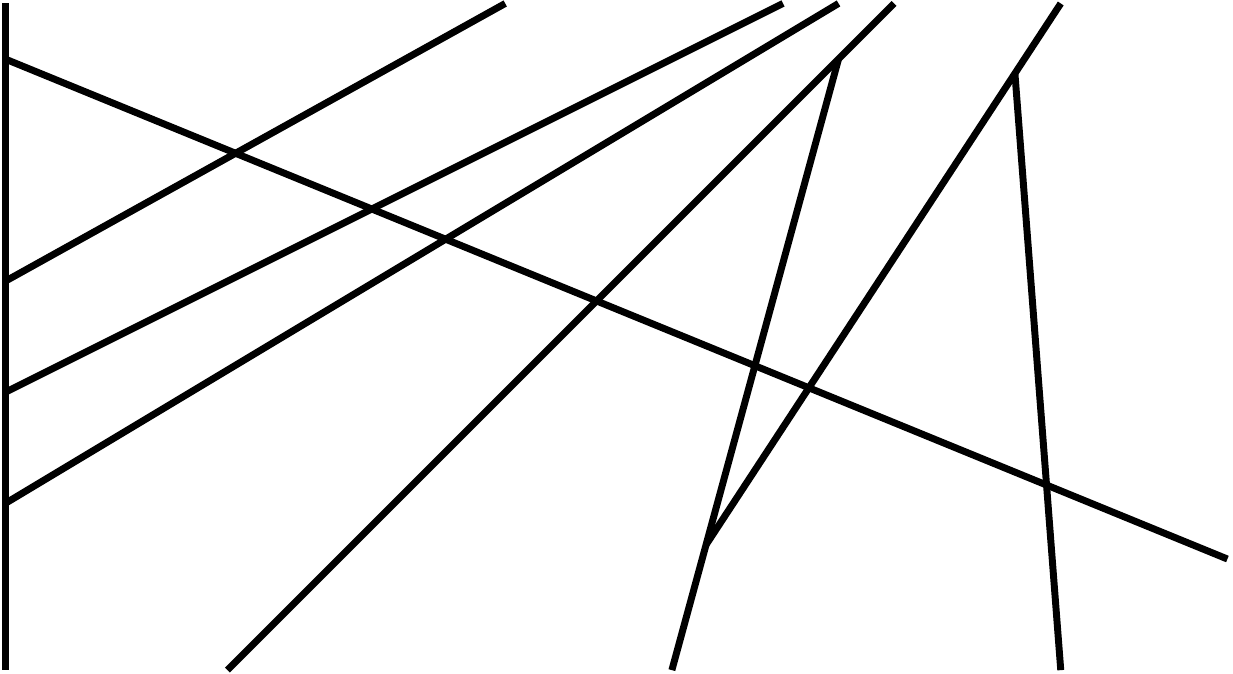}
            \caption{B\'ar\'any, Hubard, and Jer\'onimo partition}
            \label{fig:BHJdiagram}
     \end{subfigure}
        \caption{Each type of partition has its own geometry}
        \label{fig:d2diagrams}
\end{figure}

\subsection{d=3}

In the case of $d=3$, there are two other constructions that use Theorem \ref{thm:bhj} to show the same bound. Let $n = 6q + r$ for non-negative integers $q$ and $r < 6$. If $r > 0$, we use $r-1$ we first partition the measure using $r$ vertical planes so that each of the $r$ left-most regions has measure $\mu(\R^3)/n$.

For the first approach, we use a partition result of Buck and Buck, that says that \textit{any finite measure in the plane can be split into six equal parts using three concurrent lines} \cite{buck1949equipartition}.  Therefore, using a projection we can see that we can partition the rightmost region with three planes that share a common line of intersection such that each of the six resulting regions has equal measure.  Pairing every other region into a group of three gives two triplets, each consisting of well-separated regions. Therefore, by Theorem \ref{thm:bhj} we can iteratively use $q-1$ planes to partition each triplet such that the resulting regions have measure exactly $\mu(\R^3)/q$.  

In the second approach, we use two parallel planes to partition the right region into three regions of equal measure and then by the ham sandwich theorem we bisect all three regions using an additional plane. We group together the leftmost and rightmost region on the positive side of the bisecting plane and the center region on the negative side. The other triplet is formed by the remaining regions.

In both cases for $n\geq6$, the number of planes we use is
\[
\frac{n}{3} + \frac{2r}{3} + 1.
\]

\smallskip

Note that $r \leq 5$, so the complexity is bounded by $(n+13)/3$. Therefore, tools such as Theorem \ref{thm:bhj} allow us to find partitions with smaller complexity and different geometric properties.

\section{Stiefel manifolds and Yao--Yao partitions}\label{sect:stiefelyao}

The Borsuk--Ulam theorem is the topological backbone of many mass partition results.  The Yao--Yao theorem is one of such result, as the inductive step that allows us to find Yao--Yao partitions relies on the Borsuk--Ulam theorem.  As one explores more elaborate mass partition theorems, we may either require more advanced topological machinery or tailor-made topological results similar to the Borsuk-Ulam theorem.  In this section we first present a new proof a Borsuk-Ulam type theorem from Chan et al \cite{chan2020borsuk} where the domain is a Stiefel manifold of orthonormal $k$-frames in $\rr^d$
\[
V_k(\rr^d) = \{(v_1, \ldots, v_k): v_i \in S^{d-1} \mbox{ for }i=1,\ldots, k, \ \langle v_i, v_j\rangle = 0 \ \mbox{ for } i \neq j\}.
\]
The Stiefel manifold $V_k(\rr^d)$ is a manifold of dimension $(d-1) + (d-2) + \ldots + (d-k)$ with a free action of $(\zz_2)^k = \{+1,-1\}^k$.  Given an element $\lambda = (\lambda_1, \ldots, \lambda_k) \in (\zz_2)^k$ and $v=(v_1, \ldots, v_k) \in V_k(\rr^d)$, we define
\[
\lambda v := (\lambda_1 v_1, \ldots, \lambda_k v_k) \in V_k(\rr^d).
\]
We can also define an action of $(\zz_2)^k$ on $R=\rr^{d-1} \times \rr^{d-2} \times \ldots \times \rr^{d-k}$ as the direct product of the action of $\zz_2$ on each component.  The only fixed point in $R$ is the zero vector.

\begin{theorem}[Chan, Chen, Frick, Hull 2020 \cite{chan2020borsuk}]\label{thm:stiefel-BU}
	Let $k, d$ be positive integers such that $k \le d$.  Let $f:V_k(\rr^d) \to \rr^{d-1} \times \rr^{d-2} \times \ldots \times \rr^{d-k}$ be a continuous and equivariant function with respect to the action of $(\zz_2)^k$ on each space as defined above.  Then, there exists $v \in V_k(\rr^d)$ such that $f(v) = 0$.
\end{theorem}

The case $k=1$ is one of the many equivalent forms of the Borsuk-Ulam theorem: \textit{every continuous odd map $f:S^{d-1}\to \rr^{d-1}$ has a zero.}  The case $k=d-1$ is essentially a Borsuk-Ulam type theorem in which the domain is $O(d)$.  The original proof of this theorem involves the computation a topological invariant constructed by the sum of the degrees of some associated continuous maps.  It can also be proved using the Fadell--Husseini index (this proof is implicit in Fadell and Husseini's foundational paper \cite{Fadell:1988tm}).  Similar results, where the domain is $O(d)$ or $SO(d)$, have been proven earlier, although they use slightly different group actions \cites{roldan2014extension, FHMRA19}.  Since the dimension of the domain and the range is the same, the theorem above is also a consequence of Musin's Borsuk--Ulam type theorems for manifolds \cite{Mus12}*{Theorem 1}.  We present a simple proof below for completeness.  The proof below, as well as Musin's proof for this main theorem and the proofs for the results mentioned for $SO(d)$ all follow the scheme of Imre B\'ar\'any's proof of the Borsuk--Ulam theorem \cites{barany1980borsuk, matousek2003using}.

\begin{proof}
First, we construct a particular map $g:V_k(\rr^d) \to R$ that is continuous and equivariant.  We denote the coordinates of each $v_i$ by $v_i = ((v_{i})_1, \ldots, (v_{i})_d)$.  The function $g$ is defined by
\begin{align*}
g:V_k(\rr^d) &\to \rr^{d-1}\times \rr^{d-2} \times \ldots \times \rr^{d-k} \\ 
(v_1, \ldots, v_k) &\mapsto (x_1, \ldots, x_k)	
\end{align*}
where
\[
x_j = \begin{bmatrix}
	(v_j)_{j+1} \\
	(v_j)_{j+2} \\
	\vdots \\
	(v_j)_{d}
\end{bmatrix} \in \rr^{d-j}.
\]
A simple inductive argument shows that the only zeroes of this function are when for each $1\le j \le k$ we have $v_j \in \{e_j, -e_j\}$, where $(e_1,\ldots, e_d)$ is the canonical basis of $\rr^d$.  In other words, there is a single $\left(\zz_2 \right)^k$-orbit of zeros in $V_k(\rr^d)$, which is the set $\left(\zz_2 \right)^k(e_1,\ldots, e_k)$.

Now we construct a new map
\begin{align*}
	T: V_k(\rr^d) \times [0,1] &\to \rr^{d-1}\times \rr^{d-2} \times \ldots \times \rr^{d-k} \\
	(v,t) & \mapsto tf(v) + (1-t)g(v).
\end{align*}

Let $\varepsilon>0$ be a real number. There exists a smooth $(\zz_2)^k$-equivariant map $T_{\varepsilon}: V_k(\rr^d) \times [0,1] \to \rr^{d-1}\times \rr^{d-2} \times \ldots \times \rr^{d-k}$ such that
\begin{align*}
T_{\varepsilon}(v,0) = T(v,0) = g(v) & \qquad \mbox{ for all }v\in V_k(\rr^d), \\
||T_{\varepsilon}(v,t)-T(v,t)||< \varepsilon & 	\qquad \mbox{ for all }v\in V_k(\rr^d),\ t \in [0,1], \mbox{ and} \\
0\mbox{ is a regular value of }T_{\varepsilon}.
\end{align*}

This follows from Thom's transversality theorem \citelist{\cite{thom1954quelques} \cite{guillemin2010differential}*{pp 68-69}}.  Let us look at $T_{\varepsilon}^{-1}(0) \subset V_k(\rr^d)\times [0,1]$.  This is a one-dimensional manifold with a free action of $\left( \zz_2\right)^k$, whose connected components are diffeomorophic to circles or to intervals.  The components with diffeomorphic to intervals must have their endpoints in $V_k(\rr^d) \times \{0,1\}$.  As any continuous function from a closed interval to itself must have a fixed point, the group $(\zz_2)^k$ acts freely on the set of intervals in $T^{-1}_{\varepsilon}(0)$.

The set $T^{-1}_{\varepsilon}(0) \cap \left(V_k(\rr^d)\times \{0\}\right)$ has a single orbit of $\left( \zz_2\right)^k$ by construction.  Therefore, $T^{-1}_{\varepsilon}(0) \cap \left(V_k(\rr^d)\times \{1\}\right)$ must have an odd number of orbit of $\left( \zz_2\right)^k$, which implies that $T^{-1}_{\varepsilon}(0) \cap \left(V_k(\rr^d)\times \{1\}\right)$ is not empty.  As we make $\varepsilon \to 0$, the compactness of $V_k(\rr^d)$ implies that $f^{-1}(0)$ is not empty, as we wanted.
\end{proof}

A consequence of Theorem \ref{thm:stiefel-BU} is a generalization of the central transversal theorem mentioned in the introduction \cites{dol1992generalization, zivaljevic1990extension}.   Given a finite probability measure $\mu$ in $\rr^d$ and an affine subspace $L$ of dimension $k$, we say that $L$ is a central $k$-transversal to $\mu$ if each half-space that contains $L$ has measure at least $1/(d-k+1)$ in $\mu$.  The central transversal theorem says that \textit{for $0 \le \lambda \le d-1$, any $\lambda+1$ measures in $\rr^d$ have a common central $\lambda$-transversal}.  The case $\lambda=d-1$ is the ham sandwich theorem and the case $\lambda=0$ is the centerpoint theorem.  We obtain the following result.

\begin{theorem}\label{thm:stiefel-transversal}
	Let $d$ be a positive integer, $\mu_1, \ldots, \mu_d$ be finite absolutely continuous measures in $\rr^d$, and $\lambda$ be an integer such that $0 \le \lambda \le d-1$.  We can find affine subspaces $L_{\lambda} \subset L_{\lambda+1} \subset \ldots \subset L_{d-1}$ such that
	\begin{itemize}
		\item for each $\lambda \le i \le d-1$, the dimension of $L_i$ is $i$,
		\item the subspace $L_{\lambda}$ is a central $\lambda$-transversal to each of $\mu_1, \ldots, \mu_{\lambda+1}$, and
		\item for each $\lambda < i \le d-1$, the subspace $L_i$ is a central $i$-transversal to $\mu_{i+1}$.
	\end{itemize}
\end{theorem}

The second condition implies the central transversal theorem.  The last condition can also be obtained from the central transversal theorem by a bootstrapping argument.  We first find a common $\lambda$-transversal to the first $\lambda+1$ measures and then use the central transversal theorem to iteratively look for a direction to extend $L_n$ to $L_{n+1}$ for $\lambda \le n < d-1$.  This version is slightly closer to the ham sandwich theorem than the central transversal theorem, as we always deal with $d$ measures.  We use Theorem \ref{thm:stiefel-BU} to obtain a direct proof.

\begin{proof}
	Let $k=d-\lambda$.  We construct a function $f:V_k(\rr^d) \to R$ as in Theorem \ref{thm:stiefel-BU}.  Let $(v_1, \ldots, v_k)$ be an element of $V_k(\rr^d)$.  For $0 \le i \le \lambda=d-k$, let $\sigma_{i+1}$ be the orthogonal projection of $\mu_{i+1}$ onto $V_{k} = \operatorname{span}\{v_1, \ldots, v_k\}$.  By the central point theorem, there is a centerpoint $p_{i+1}$ of $\sigma_{i+1}$ in $V_k$.  The set of all possible centerpoints of $\sigma_i$ is a compact convex set, so we can choose $p_{i+1}$ to be the barycenter of this set.  If $\Pi_k : \rr^d \to V_k$ is the orthogonal projection, the affine subspace $\Pi_k^{-1}(p_{i+1})$ is a central $\lambda$-transversal for $\mu_{i+1}$.
	
	For $\lambda < i \le d-1$, let $\sigma_{i+1}$ be the orthogonal projection of $\mu_{i+1}$ onto $V_{d-i}=\operatorname{span} \{ v_1, \ldots, v_{d-i}\}$.  Let $p_{i+1}$ be the centerpoint of $\sigma_i$ in $V_{d-i}$ chosen as above. We denote by $\Pi_{d-i}: \rr^d \to V_{d-i}$ the orthogonal projection onto $V_{d-i}$.  The affine subspace $\Pi^{-1}_{d-i}(p_{i+1})$ is a central $i$-transversal for $\mu_{i+1}$.
	Now we define
	\begin{align*}
		f:V_k(\rr^d) &\to \rr^{d-1} \times \rr^{d-2} \times \ldots \rr^{d-k} \\
		(v_1, \ldots, v_k) &\mapsto (x_1, \ldots, x_k)
	\end{align*}
	where $x_j \in \rr^{d-j}$ is defined by 
	\[
	x_j = \begin{bmatrix}
		\langle v_j , p_2 - p_1 \rangle \\
		\langle v_j , p_3 - p_1 \rangle \\
		\vdots \\
		\langle v_j , p_{d+1-j}-p_1\rangle
	\end{bmatrix}.
	\]
	This map is continuous and equivariant, so it must have a zero.  If $(v_1, \ldots, v_k)$ is a zero of this map, let us show that the subspaces $L_i = \Pi_{d-i}^{-1}\left(\Pi_{d-i}(p_1)\right)$ for $\lambda \le i \le d-1$ satisfy the condition we want.  We immediately get that the dimension of $L_i$ is equal to $i$ and that $L_\lambda \subset \ldots \subset L_{d-1}$.
	
	First, since $(v_1, \ldots, v_k)$ is an orthonormal frame and $p_{i+1} \in V_k$ for $i=0,\ldots, \lambda$, we know that $p_{i+1} = \sum_{j=1}^k \langle v_j, p_{i+1} \rangle v_j$.  For $i \ge \lambda+1$, we have $p_{i+1} = \sum_{j=1}^{d-i}\langle v_j, p_{i+1} \rangle v_j$.
	
	If $x_j = 0$ for all $j$, this implies that $p_1 = p_i$ for $i=2,\ldots, \lambda+1$.  Therefore, $L_{\lambda}$ is a central $\lambda$-transversal for each of $\mu_1, \ldots, \mu_{d+1}$.  For $i>\lambda+1$ we have
	\[
	\Pi_{d-i}(p_1) = \sum_{j=1}^{d-i}\langle v_j, p_1 \rangle v_j = \sum_{j=1}^{d-i}\langle v_j, p_{i+1} \rangle v_j = p_{i+1}
	\].
	Therefore, $L_i$ is a central $i$-transversal to $\mu_{i+1}$.
\end{proof}

If we take $\lambda = 0$ we have the following corollary about full flags of subspaces.

\begin{corollary}\label{cor:flag}
	Let $\mu_1, \ldots, \mu_d$ be finite measures in $\rr^d$.  We can find affine subspaces $L_{0} \subset L_{1} \subset \ldots \subset L_{d-1}$ such that for each $0 \le i \le d-1$, the dimension of $L_i$ is $i$ and  for each $0 \le i \le d-1$ the subspace $L_i$ is a central $i$-transversal to $\mu_{i+1}$.
\end{corollary}

The fact that our parametrization depends on $V_k(\rr^d)$ instead of a Grassman manifold means that in the proof above we may replace $p_i$ in the proof above by a point that depends continuously on the choice of $V_k(\rr^d)$, as long as the choice keeps our function equivariant.  For example, we could take $p_i$ to be a Yao--Yao center of $\sigma_{i+1}$ on the corresponding subspace with the basis induced by $(v_1,\ldots, v_k)$.  If we do this for several measures, even though the projections of the Yao--Yao centers will coincide we have no guarantee that the projections of the whole Yao--Yao partitions will coincide.  We can now prove Theorem \ref{thm:two-hyperplanes}.

\begin{proof}[Proof of Theorem \ref{thm:two-hyperplanes}]
The proof is almost identical to the proof of Theorem \ref{thm:stiefel-transversal} with $k=2$ except for the definition of $p_1$.

We define $p_1$ as the Yao--Yao center of $\Pi_2(\mu_1)$ induced by the basis $(v_2, v_1)$.  Notice that a Yao--Yao partition in $\rr^2$ is an equipartition by two lines, one of which is parallel to $v_2$ if we choose the ordered basis as above.  The inverse image under $\Pi_2$ of the two lines forming this Yao--Yao partition give us the hyperplanes we wanted.
\end{proof}

Let us compare Theorem \ref{thm:two-hyperplanes} with earlier results.  For $d=3$ the central transversal theorem tells us that for any two measures we can find a common central line.  A classic result of Hadwiger tells us that for any two measures in $\rr^3$ there are two planes that simultaneously split them into four equal parts \cite{Hadwiger1966} (a proof with new methods was recently found by Blagojevi\'c, Frick, Haase, and Ziegler \cite{Blagojevic:2018jc}*{Section 4}).

The case $d=3$ of Theorem \ref{thm:two-hyperplanes} shows that for any two measure in $\rr^3$ we can find two planes that split the first measure into four equal parts and whose intersection is a central line for the second measure.  We even have a degree of freedom since we can choose $\mu_3$ at will.  Even the case $\mu_1 = \mu_2$ is non-trivial, as opposed to the previous two related results.

If we are given more than $2d/3$ measures in $\rr^d$ it is possible that there is no pair of hyperplanes that split each of them into four equal parts \cite{Ramos:1996dm}.  Therefore the corollary above is a sensible way of interpolating between central transversals and equipartitions by hyperplanes of many measures.  It is not clear if we can make the equipartition part of Theorem \ref{thm:two-hyperplanes} be used for more measures.

\begin{question}
	Is is true that for any $d-1$ measures $\mu_1, \ldots, \mu_{d-1}$ in $\rr^d$ there exist two hyperplanes $H_1, H_2$ such that
	\begin{itemize}
		\item $H_1 \cup H_2$ split each of $\mu_1, \mu_2$ into four equal parts and
		\item $H_1 \cap H_2$ is a central $(d-2)$ transversal to each of $\mu_3, \ldots, \mu_{d-1}$?
	\end{itemize}
\end{question}

We can also extend Corollary \ref{cor:flag} in a similar way but now using a full Yao--Yao partition.

\begin{theorem}\label{thm:yao-transversal}
Let $d$ be a positive integer and $\mu_1, \ldots, \mu_d$ be finite absolutely continuous measures in $\rr^d$.  For an orthonormal basis $(u_1, \ldots, u_d)$ of $\rr^d$ let $C(\mu_1)$ be the center of the Yao--Yao partition of $\mu_1$ induced by the basis $(u_1, \ldots, u_d)$, and let $L_i$ be the translate of $\operatorname{span}\{u_1,\ldots, u_i\}$ through $C(\mu_1)$.  Then, there exists a choice of an orthonormal basis such that for $i=1,\ldots,d-1$ we have that $L_i$ is a central $i$-transversal of $\mu_{i+1}$.
\end{theorem}

Due to Lemma \ref{lem:skeleton}, each $L_i$ is contained in the union of the $i$-skeletons of the parts of the Yao--Yao partition of $\mu_1$ we constructed, so these spaces appear naturally.

\begin{proof}
We follow verbatim the proof of Theorem \ref{thm:stiefel-transversal} with the only difference being that for each $v=(v_1, \ldots, v_d) \in V_d(\rr^d)$ we take $p_1$ to be the Yao--Yao center of $\mu_1$ according to the basis $(u_1, u_2, \ldots, u_d) = (v_d, v_{d-1}, \ldots, v_1)$.  For $i=1,\ldots, d-1$, the point $p_{i+1}$ is still the centerpoint of the projection of $\mu_{i+1}$ onto $\operatorname{span}\{v_1,\ldots, v_{d-i}\}$.
\end{proof}

Just as the theorem above is closely related to Corollary \ref{cor:flag}, it is clear we can get an analogous extension of Theorem \ref{thm:stiefel-transversal} for any $\lambda$.  We do this by choosing $p_1$ to be the Yao--Yao center of the projection of $\mu_1$ onto $\operatorname{span}\{v_1, \ldots, v_k\}$ induced by the basis $(v_k,\ldots, v_1)$.  Any half-space that contains $L_{\lambda}$ would avoid the interior of one of the regions of the initial Yao--Yao partition constructed.

\section{Additional remarks}\label{sect:remarks}

Theorem \ref{thm:generalyao} gives an exact bound for $d=1$, but it is unclear whether this is the best that can be done for $d\geq2$. The only lower bound for Problem \ref{problem:generalizedyao} was proven by Rold\'an-Pensado and Sober\'on who showed that $\displaystyle n \geq 2^{\frac{d}{2} - 1}$ for $t=1$ \cite{roldan2014extension}. Besides this result, the following question remains open and relatively unexplored. 

\begin{problem}\label{problem:lowerboundn}
Let $t, d$ be positive integers. Let $n$ be the smallest value such that for any finite measure $\mu$ in $\R^d$ that is absolutely continuous with respect to the Lebesgue measure there exists a convex equipartition into $n$ parts such that every hyperplane avoids the interior of at least $t$ regions. What is the lower bound for $n$?
\end{problem}

In the case of $d=2$, we offer multiple constructions that match the current best upper bound. Looking at the complexity of a partition is not sufficient to give us a lower bound on $n$ in general.  However, it allows us to bound $t$ in terms of $n$. 

\begin{problem}\label{problem:greatesttforn}
Given $n, d$ find the largest $t$ such that for any finite measure $\mu$ in $\R^d$ that is absolutely continuous with respect to the Lebesgue measure there exists a convex equipartition into $n$ parts such that every line misses the interior of at least $t$ parts.
\end{problem}

Let $k$ be the complexity for a convex partition of $\rr^d$.  As we discussed in the previous sections, an upper bound on the complexity of an equipartition into $n$ parts gives us a lower bound on $t$ because $(n - (k + 1)) \leq t$. The following problem becomes relevant.

\begin{problem}\label{problem:exactk}
What is the smallest value of $k$ such that for every finite absolutely continuous measure $\mu$ in $\rr^d$ there exists a convex equipartition of $\mu$ with complexity at most $k$?
\end{problem}

Equipartitions impose greater geometric restrictions on possible partitions as Buck and Buck showed that there is no equipartition of $7$ regions of complexity $3$ \cite{buck1949equipartition}. Therefore, it seems that in order to prove an exact bound on $k$ the geometry of the equipartitions should be considered.

It also seems that the original construction of Yao and Yao can be used to obtain more mass partition results.  For example, consider the following proof of a special case of Theorem \ref{thm:two-hyperplanes} for $d=3$.

\begin{lemma}
Let $\mu_1, \mu_2$ be two finite absolutely continuous measures in $\rr^3$ whose supports can be separated by a plane $H$.  Then, there exist two planes $H_1, H_2$ such that $H_1 \cup H_2$ splits $\mu_1$ into four equal parts and $H_1 \cap H_2$ is a central line for $\mu_2$.
\end{lemma}

Before showing the proof, notice the following property of the result above.  Given a half-space $M$, if the bounding plane of $M$ hits $H_1 \cap H_2$ on the side of $H$ of $\mu_1$, then $\mu_2(M) \ge \mu_2(\rr^3)/3$.  If it hits $H_1 \cap H_2$ on the side of $H$ of $\mu_2$, then $M$ contains one of the four regions that of the equipartition of $\mu_1$.

\begin{proof}
Assume without loss of generality that $H$ his a horizontal plane, and choose a basis $u_1, u_2$ of $H$ arbitrarily.  For a vector $v \not\in \operatorname{span}\{u_1, u_2\}$, let $C_v(\mu_1)$ be the Yao--Yao center of $p_v(\mu_1)$ with respect to $u_1, u_2$.   Let $C_v(\mu_2)$ be the centerpoint of $p_v(\mu_2)$.

The exact same arguments of Yao and Yao show that, up to scalar multiplication, there exists a unique vector $v$ for which $C_v(\mu_1) = C_v(\mu_2)$ (the separation of the supports of $\mu_1, \mu_2$ by $H$ is needed for this).  The Yao--Yao partition of $p_v(\mu_1)$ with respect to $u_1, u_2$ consists of two lines.  If we extend these two lines by the direction $v$, we obtain $H_1$ and $H_2$.
\end{proof}

As a final remark, we compare Corollary \ref{coro-line-miss-many} to a similar problem regarding Yao--Yao type partitions for more measures.

\begin{problem}[Problem 3.3.3 in \cite{roldan2021survey} for lines]
Let $d$ be a positive integer.  Find the smallest positive integer $n$ such that the following holds.  For any $d-1$ absolutely continuous probability measures $\mu_1, \dots, \mu_{d-1}$ in $\rr^d$ there exists a convex partition $C_1, \dots C_n$ of $\rr^d$ such that
	\[
	\mu_i (C_j) = \frac{1}{n} \qquad \text{for } i=1,\dots, d-1 \quad j= 1,\dots, n
	\]
	and every line misses the interior of at least one $C_j$.
\end{problem}

Corollary \ref{coro-line-miss-many} shows that the problem above is meaningful even if we have few measures.  The problem above shows that we might expect results of this kind for up to $d-1$ measures simultaneously.  Determining how the number of regions we are guaranteed to miss with a line decreases as we increase the number of measures in the equipartition is an interesting problem.


\begin{bibdiv}
\begin{biblist}

\bib{alon2005crossing}{article}{
      author={Alon, Noga},
      author={Pach, J{\'a}nos},
      author={Pinchasi, Rom},
      author={Radoi{\v{c}}i{\'c}, Rado{\v{s}}},
      author={Sharir, Micha},
       title={Crossing patterns of semi-algebraic sets},
        date={2005},
     journal={Journal of Combinatorial Theory, Series A},
      volume={111},
      number={2},
       pages={310\ndash 326},
}

\bib{barany1980borsuk}{article}{
      author={B{\'a}r{\'a}ny, Imre},
       title={Borsuk's theorem through complementary pivoting},
        date={1980},
     journal={Mathematical Programming},
      volume={18},
      number={1},
       pages={84\ndash 88},
}

\bib{buck1949equipartition}{article}{
      author={Buck, Robert~C},
      author={Buck, Ellen~F},
       title={Equipartition of convex sets},
        date={1949},
     journal={Mathematics Magazine},
      volume={22},
      number={4},
       pages={195\ndash 198},
}

\bib{Blagojevic:2018jc}{article}{
      author={Blagojevi{\'c}, Pavle V.~M.},
      author={Frick, Florian},
      author={Haase, Albert},
      author={Ziegler, G{\"u}nter~M.},
       title={Topology of the {G}r{\"u}nbaum--{H}adwiger--{R}amos hyperplane
  mass partition problem},
        date={2018},
     journal={Transactions of the American Mathematical Society},
      volume={370},
      number={10},
       pages={6795\ndash 6824},
}

\bib{barany2008slicing}{article}{
      author={B{\'a}r{\'a}ny, Imre},
      author={Hubard, Alfredo},
      author={Jer{\'o}nimo, Jes{\'u}s},
       title={Slicing convex sets and measures by a hyperplane},
        date={2008},
     journal={Discrete \& Computational Geometry},
      volume={39},
      number={1-3},
       pages={67\ndash 75},
}

\bib{barba2019sharing}{article}{
      author={Barba, Luis},
      author={Pilz, Alexander},
      author={Schnider, Patrick},
       title={Sharing a pizza: bisecting masses with two cuts},
        date={2019},
     journal={arXiv preprint arXiv:1904.02502},
}

\bib{Blagojevic:2018gt}{article}{
      author={Blagojevi{\'c}, Pavle V.~M.},
      author={Sober{\'o}n, Pablo},
       title={Thieves can make sandwiches},
        date={2018},
     journal={Bulletin of the London Mathematical Society},
      volume={50},
      number={1},
       pages={108\ndash 123},
}

\bib{chan2020borsuk}{article}{
      author={Chan, Yu~Hin},
      author={Chen, Shujian},
      author={Frick, Florian},
      author={Hull, J~Tristan},
       title={{Borsuk-Ulam theorems for products of spheres and Stiefel
  manifolds revisited}},
        date={2020},
     journal={Topological Methods in Nonlinear Analysis},
      volume={55},
      number={2},
       pages={553\ndash 564},
}

\bib{deLongueville:2006uo}{article}{
      author={De~Longueville, Mark},
      author={{\v Z}ivaljevi{\'c}, Rade~T.},
       title={Splitting multidimensional necklaces},
        date={2008},
     journal={Advances in Mathematics},
      volume={218},
      number={3},
       pages={926\ndash 939},
}

\bib{dol1992generalization}{article}{
      author={Dol'nikov, Vladimir~Leonidovich},
       title={A generalization of the ham sandwich theorem},
        date={1992},
     journal={Mathematical Notes},
      volume={52},
      number={2},
       pages={771\ndash 779},
}

\bib{Fadell:1988tm}{article}{
      author={Fadell, Edward},
      author={Husseini, Sufian},
       title={An ideal-valued cohomological index theory with applications to
  {B}orsuk--{U}lam and {B}ourgin--{Y}ang theorems},
        date={1988},
     journal={Ergodic Theory and Dynamical Systems},
      volume={8},
       pages={73\ndash 85},
}

\bib{FHMRA19}{article}{
      author={Fradelizi, Matthieu},
      author={Hubard, Alfredo},
      author={Meyer, Mathieu},
      author={R{old{\'a}n-Pensado}, Edgardo},
      author={Zvavitch, Artem},
       title={{Equipartitions and Mahler volumes of symmetric convex bodies}},
        date={2019},
     journal={arXiv preprint arXiv:1904.10765},
}

\bib{guillemin2010differential}{book}{
      author={Guillemin, Victor},
      author={Pollack, Alan},
       title={Differential topology},
   publisher={American Mathematical Society},
        date={2010},
      volume={370},
}

\bib{grunbaum1960partitions}{article}{
      author={Gr{\"u}nbaum, Branko},
      author={others},
       title={Partitions of mass-distributions and of convex bodies by
  hyperplanes.},
        date={1960},
     journal={Pacific Journal of Mathematics},
      volume={10},
      number={4},
       pages={1257\ndash 1261},
}

\bib{Hadwiger1966}{article}{
      author={Hadwiger, Hugo},
       title={Simultane {V}ierteilung zweier {K}{\"o}rper},
        date={1966},
     journal={Archiv der Mathematik},
      volume={17},
      number={3},
       pages={274\ndash 278},
}

\bib{Hubard:2019we}{article}{
      author={Hubard, Alfredo},
      author={Karasev, Roman~N.},
       title={Bisecting measures with hyperplane arrangements},
        date={2020},
     journal={Mathematical Proceedings of the Cambridge Philosophical Society},
      volume={169},
      number={3},
       pages={639\ndash 647},
}

\bib{Karasev:2016cn}{article}{
      author={Karasev, Roman~N.},
      author={Rold{\'a}n-Pensado, Edgardo},
      author={Sober{\'o}n, Pablo},
       title={Measure partitions using hyperplanes with fixed directions},
        date={2016},
     journal={Israel journal of mathematics},
      volume={212},
      number={2},
       pages={705\ndash 728},
}

\bib{lehec2009yao}{article}{
      author={Lehec, Joseph},
       title={{On the Yao--Yao partition theorem}},
        date={2009},
     journal={Archiv der Mathematik},
      volume={92},
      number={4},
       pages={366\ndash 376},
}

\bib{matousek2003using}{book}{
      author={Matou{\v s}ek, Ji{\v{r}}{\'\i}},
       title={Using the {B}orsuk--{U}lam theorem: lectures on topological
  methods in combinatorics and geometry},
   publisher={Springer},
        date={2003},
}

\bib{Mus12}{article}{
      author={Musin, Oleg},
       title={{Borsuk--Ulam type theorems for manifolds}},
        date={2012},
     journal={Proceedings of the American Mathematical Society},
      volume={140},
      number={7},
       pages={2551\ndash 2560},
}

\bib{Ramos:1996dm}{article}{
      author={Ramos, Edgar~Arturo},
       title={Equipartition of mass distributions by hyperplanes},
        date={1996},
     journal={Discrete {\&} Computational Geometry},
      volume={15},
      number={2},
       pages={147\ndash 167},
}

\bib{roldan2014extension}{article}{
      author={R{old{\'a}n-Pensado}, Edgardo},
      author={Sober{\'o}n, Pablo},
       title={{An extension of a theorem of Yao and Yao}},
        date={2014},
     journal={Discrete \& Computational Geometry},
      volume={51},
      number={2},
       pages={285\ndash 299},
}

\bib{roldan2021survey}{article}{
      author={R{old{\'a}n-Pensado}, Edgardo},
      author={Sober{\'o}n, Pablo},
       title={A survey of mass partitions},
        date={2021},
     journal={Bulletin of the American Mathematical Society},
        note={Electronically published on February 24, 2021, DOI:
  https://doi.org/10.1090/bull/1725 (to appear in print).},
}

\bib{Steinhaus1938}{article}{
      author={Steinhaus, Hugo},
       title={A note on the ham sandwich theorem},
        date={1938},
     journal={Mathesis Polska},
      volume={9},
       pages={26\ndash 28},
}

\bib{thom1954quelques}{article}{
      author={Thom, Ren{\'e}},
       title={Quelques propri{\'e}t{\'e}s globales des vari{\'e}t{\'e}s
  diff{\'e}rentiables},
        date={1954},
     journal={Commentarii Mathematici Helvetici},
      volume={28},
      number={1},
       pages={17\ndash 86},
}

\bib{yao1985general}{inproceedings}{
      author={Yao, Andrew~C},
      author={Yao, F~Frances},
       title={A general approach to d-dimensional geometric queries},
        date={1985},
   booktitle={{Proceedings of the seventeenth annual ACM symposium on Theory of
  Computing}},
       pages={163\ndash 168},
}

\bib{Zivaljevic:2017vi}{incollection}{
      author={{\v{Z}}ivaljevi{\'c}, Rade~T.},
       title={Topological methods in discrete geometry},
        date={2017},
   booktitle={Handbook of discrete and computational geometry, third edition},
   publisher={CRC Press},
}

\bib{zivaljevic1990extension}{article}{
      author={{\v{Z}}ivaljevi{\'c}, Rade~T},
      author={Vre{\'c}ica, Sini{\v{s}}a~T},
       title={An extension of the ham sandwich theorem},
        date={1990},
     journal={Bulletin of the London Mathematical Society},
      volume={22},
      number={2},
       pages={183\ndash 186},
}

\end{biblist}
\end{bibdiv}
\end{document}